\documentclass[11pt,a4paper]{amsart}
\usepackage[T1]{fontenc}
\usepackage[utf8]{inputenc}
\usepackage{lmodern}
\usepackage[margin=25mm]{geometry}
\usepackage{amsmath,amssymb,amsthm}
\usepackage{mathtools}
\usepackage{enumitem}
\usepackage{microtype}
\usepackage{xcolor}
\usepackage[colorlinks=true,linkcolor=blue!45!black,citecolor=blue!45!black,urlcolor=blue!45!black]{hyperref}

\hypersetup{
  pdftitle={Finite-Point Metrizable Coarsenings: Compatible Gauges, Simplicial Metrics, and Hausdorff Lower Bounds},
   pdfauthor={First Author Name; Second Author Name},
  pdfsubject={Finite-point localized coarsenings and the compactness--lattice dichotomy}
}

\newtheorem{theorem}{Theorem}[section]
\newtheorem{lemma}[theorem]{Lemma}
\newtheorem{proposition}[theorem]{Proposition}
\newtheorem{corollary}[theorem]{Corollary}
\theoremstyle{definition}
\newtheorem{definition}[theorem]{Definition}
\newtheorem{example}[theorem]{Example}
\theoremstyle{remark}
\newtheorem{remark}[theorem]{Remark}
\newcommand{\N}{\mathbb N}
\newcommand{\R}{\mathbb R}
\newcommand{\LF}{\mathcal L_F^m(\tau)}
\newcommand{\GF}{\mathcal G_F(\tau)}
\newcommand{\Del}{\Delta^{k-1}}
\setlist[enumerate,1]{label=\textup{(\roman*)},leftmargin=2.5em,itemsep=3pt,topsep=5pt}
\title[Finite-Point Metrizable Coarsenings]{Finite-Point Metrizable Coarsenings:\\
Compatible Gauges, Simplicial Metrics, and\\
Hausdorff Lower Bounds}

\author{Ahmad ja'afari kalvan$^{*}$}
\address{Department of Mathematics, Tarbiat Modares University, 14115-134, Tehran, Iran}
\email{ahmad.jaafari@modares.ac.ir}
\thanks{$^{*}$Corresponding author}

\author{Ehsan Shahoseini$^{1}$}
\address{School of Mathematics, Institute for Research in Fundamental Sciences (IPM), P.O. Box: 19395-5746, Tehran, Iran}
\email{shahoseini@ipm.ir}
\thanks{$^{1}$ The second author's research was supported by a grant from IPM.}

\date{}

\subjclass[2020]{Primary 54A10; Secondary 54E35, 54E50, 54D35,54D30}
\keywords{Coarser topology, finite-point localization, continuous gauge, simplex metric,
Hausdorff lower bound, lattice of topologies}

\begin{document}

\begin{abstract}
Let $(X,\tau)$ be metrizable and let $F=\{a_1,\ldots,a_k\}\subseteq X$, where
$2\le k<\infty$. We represent all metrizable topologies $\sigma\subseteq\tau$
agreeing with $\tau$ on $X\setminus F$ by compatible systems of continuous
gauges $s_i:X\to[0,1]$ with $s_i^{-1}(0)=\{a_i\}$.
The condition $\inf_X\max\{s_i,s_j\}>0$ for $i\ne j$ is equivalent to both
Hausdorffness and metrizability of the prescribed gauge topology.
A normalized product map into the standard simplex gives an explicit metric;
its triangle inequality follows from a simplex slack inequality.
This metric is complete whenever the auxiliary bounded compatible metric is complete.

For two compatible systems, their coordinatewise minimum describes the
intersection topology. It is compatible exactly when the two coarsenings
have a common Hausdorff lower bound; in that case the intersection is
metrizable and is their meet. Otherwise every common lower topology is
non-Hausdorff. A closed-discrete construction produces such an obstructed
pair for every noncompact metrizable space and every finite exceptional
set with at least two points. Consequently, for these exceptional sets,
the family is downward directed, or is a lattice, if and only if $(X,\tau)$
is compact, in which case it consists only of $\tau$.
\end{abstract}

\maketitle

\section{Introduction}
Let $(X,\tau)$ be metrizable and let $F=\{a_1,\ldots,a_k\}\subseteq X$ be finite.
We study the family $\LF$ of metrizable topologies $\sigma$ satisfying
\[
  \sigma\subseteq\tau,\qquad
  \sigma|_{X\setminus F}=\tau|_{X\setminus F}.
\]
For a single exceptional point, continuous gauges and a cone metric give
a representation with cofinal order and finite maximum/minimum formulas
\cite[Propositions~3.2--3.4, Theorem~3.5, and Corollary~3.6]{JS}.
For $k\ge2$, the local bases at different centers must also be compatible.
Our main results give an explicit metric realizing them simultaneously
and identify the compatibility condition governing intersections.

For continuous gauges $s_i:X\to[0,1]$ with $s_i^{-1}(0)=\{a_i\}$,
the exact separation condition is
\[
  \inf_{x\in X}\max\{s_i(x),s_j(x)\}>0\qquad(i\ne j).
\]
Under this condition, a normalized product map $U_s:X\to\Del$ sends the
exceptional points exactly to the vertices. We combine the pullback of the
$\frac12\ell^1$ metric with a bounded compatible metric on $X$, weighted
by distance from the vertex set. Lemma~\ref{lem:slack} supplies the
triangle-inequality estimate. Theorem~\ref{thm:metric} identifies the resulting
topology and proves completeness of this metric when the auxiliary metric
is complete; Theorem~\ref{thm:representation} represents every member of $\LF$.

Theorem~\ref{thm:meet} gives the corresponding intersection criterion.
Coordinatewise minima describe the intersection of two gauge topologies;
compatibility of that minimum is equivalent to Hausdorffness and guarantees
metrizability. This identifies exactly when the two coarsenings have a meet
or a common Hausdorff lower bound. Failure is witnessed by a sequence
converging to different exceptional points in the two topologies.
The order argument uses the general fact that two topologies have a common
Hausdorff lower bound exactly when their intersection is Hausdorff.

The existence of such failures has a direct precedent.
Kuba \cite[Section~10, Proposition~6 and the following remark,
pp.~111--112]{Kuba2016} constructs compact metrizable coarsenings
$\theta_a$ of the Euclidean topology $\eta$, agreeing with it off $a$,
by translating a topology homeomorphic to two circles meeting at one point.
For $a\ne b$, the distinct topologies $\theta_a,\theta_b$ belong to
$\mathcal L_{\{a,b\}}^m(\eta)$ and have no common Hausdorff lower topology:
a compact Hausdorff topology has no strictly coarser Hausdorff topology.
See also the expanded treatment \cite{Kuba2026}.

A consequence of the closed-discrete construction is the universal
obstruction in Theorem~\ref{thm:universal}; it also follows by applying
\cite[Theorem~1.1]{JS} at two centers with the same prescribed discrete set.
For every noncompact metrizable space and every finite $F$ with $|F|\ge2$,
an obstructed pair exists in $\LF$. Hence $\LF$ is downward directed,
or is a lattice, exactly when $(X,\tau)$ is compact
(Corollary~\ref{cor:dichotomy}). In that case $\LF=\{\tau\}$.

These localized coarsenings sit within the broader study of weaker
metrizable topologies \cite{GTW}, finite compactifications
\cite{Magill1965,Magill1968}, lattices and metric extensions
\cite{Belnov,HJW,Koushesh,MRW,PW}, and completion or other remainder
constructions \cite{Calmutchi,CV,FGO}.
If $Y=X\setminus F$ is dense in $(X,\sigma)$, the latter is a
finite-remainder extension of $Y$ constrained by $\sigma\subseteq\tau$.
An isolated exceptional point prevents this density.
The usual extension order allows continuous maps fixing $Y$; here lower
bounds are topologies on the fixed set $X$, ordered by inclusion.
The trace comparison is detailed in Remark~\ref{rem:traces}.

Sections~2--5 develop the compatible gauges, simplicial metric, and order
criterion. Section~6 proves the compactness--lattice dichotomy.
The strictness and localized compactness criteria, and finite-change
preservation, are supporting results extending
\cite[Theorem~4.3, Corollary~4.4, and Proposition~4.5]{JS};
the completeness argument in Section~7 is the compact-remainder argument of
\cite[Theorems~2.1--2.2]{HJW}. Section~8 records further questions.

\section{Finite localized coarsenings and compatible gauges}
Throughout, $(X,\tau)$ is a metrizable space and
\[
 F=\{a_1,\ldots,a_k\}\subseteq X
\]
is a finite set of distinct points.

\begin{definition}
A \emph{finite-point localized metrizable coarsening} of $\tau$ at $F$ is a
metrizable topology $\sigma$ on $X$ such that
\[
 \sigma\subseteq\tau,\qquad \sigma|_{X\setminus F}=\tau|_{X\setminus F}.
\]
The family of all such topologies, ordered by inclusion, is denoted by $\LF$.
\end{definition}

\begin{remark}[Relation with finite-remainder extensions]
Put $Y=X\setminus F$. If $Y$ is dense in $(X,\sigma)$, then $(X,\sigma)$ is a
metrizable extension of $Y$ with finite remainder $F$. The present family is
nevertheless more restrictive than the usual extension poset because the points
of $F$ are prescribed points of an ambient space $(X,\tau)$, one requires
$\sigma\subseteq\tau$, and the topology on $Y$ is fixed in advance. If some $a_i$
is isolated in $\sigma$, then $Y$ is not dense and $(X,\sigma)$ is not an
extension of $Y$ in the usual dense-subspace sense.
\end{remark}

Since $F$ is closed in every metrizable topology on $X$, the set $X\setminus F$
is open. Hence equality of the subspace topologies implies equality of the
ambient local topologies at every point of $X\setminus F$.

For the new phenomena we assume $k\ge2$. The one-point case $k=1$ is the
theory of \cite{JS}.

\begin{definition}
A \emph{compatible gauge system} at $F$ is a $k$-tuple
\[
 s=(s_1,\ldots,s_k)
\]
of $\tau$-continuous functions $s_i:X\to[0,1]$ such that
\[
 s_i^{-1}(0)=\{a_i\}\qquad(i=1,\ldots,k)
\]
and, for every $i\ne j$,
\begin{equation}\label{eq:compat}
 c_{ij}(s):=\inf_{x\in X}\max\{s_i(x),s_j(x)\}>0.
\end{equation}
The class of compatible gauge systems is denoted by $\GF$.
\end{definition}

The compatibility condition has a direct topological meaning.
\begin{lemma}\label{lem:compat}
For functions $s_i,s_j:X\to[0,1]$ with distinct singleton zero sets, the
following are equivalent:
\begin{enumerate}
\item $\inf_X\max\{s_i,s_j\}>0$;
\item there exist $\varepsilon_i,\varepsilon_j>0$ such that
      $\{s_i<\varepsilon_i\}\cap\{s_j<\varepsilon_j\}=\varnothing$;
\item there exists $\varepsilon>0$ such that
      $\{s_i<\varepsilon\}\cap\{s_j<\varepsilon\}=\varnothing$.
\end{enumerate}
\end{lemma}
\begin{proof}
If the infimum in (i) is $c>0$, any $0<\varepsilon<c$ gives (iii).
Clearly (iii) implies (ii). Finally, if (ii) holds and
$\eta=\min\{\varepsilon_i,\varepsilon_j\}$, then at every $x$ at least one
of $s_i(x),s_j(x)$ is at least $\eta$, so the infimum in (i) is at least $\eta$.
\end{proof}

The same local prescription makes sense before compatibility is imposed.
Let $r=(r_1,\ldots,r_k)$ be any $k$-tuple of $\tau$-continuous functions
$r_i:X\to[0,1]$ with $r_i^{-1}(0)=\{a_i\}$. Define $\nu_r$ by declaring
$U\subseteq X$ to be open exactly when $U\in\tau$ and, for every $i$ with
$a_i\in U$, there is an $\varepsilon_i>0$ such that
\begin{equation}\label{eq:raw}
 \{r_i<\varepsilon_i\}\subseteq U.
\end{equation}
It is immediate from the defining condition that $\nu_r$ is a topology,
that $\nu_r\subseteq\tau$, and that the topology on $X\setminus F$ is
unchanged. The sublevel sets of $r_i$ form a neighborhood base at $a_i$.
More precisely, because $F$ is finite and $r_i(a_j)>0$ for $j\ne i$,
every sufficiently small sublevel set $\{r_i<\varepsilon\}$ avoids
$F\setminus\{a_i\}$ and is $\nu_r$-open; every sublevel set contains one
of these smaller open sublevel sets and hence is a $\nu_r$-neighborhood of $a_i$.

\begin{proposition}[Compatibility is the exact Hausdorff condition]\label{prop:hausdorff}
For a tuple $r$ as above, the topology $\nu_r$ is Hausdorff if and only if
\[
 \inf_{x\in X}\max\{r_i(x),r_j(x)\}>0\qquad(i\ne j).
\]
Thus, for these prescribed gauge topologies, pairwise compatibility is
necessary and sufficient for Hausdorff separation.
\end{proposition}
\begin{proof}
Suppose first that $\nu_r$ is Hausdorff. For $i\ne j$, choose disjoint
$\nu_r$-open neighborhoods $U_i$ of $a_i$ and $U_j$ of $a_j$.
By the definition of the local bases, there are $\varepsilon_i,\varepsilon_j>0$
such that
\[
 \{r_i<\varepsilon_i\}\subseteq U_i,\qquad
 \{r_j<\varepsilon_j\}\subseteq U_j.
\]
The two sublevel sets are disjoint, so Lemma~\ref{lem:compat} gives the
required positive lower bound.

Conversely, assume the compatibility inequalities. For distinct exceptional
points $a_i,a_j$, Lemma~\ref{lem:compat} gives disjoint sublevel neighborhoods;
shrinking their levels if necessary, we may also ensure that each avoids $F$
except for its own center, and hence both are $\nu_r$-open.
If $a_i\in F$ and $x\notin F$, choose
\[
 0<\varepsilon<
 \min\left\{\frac{r_i(x)}2,\frac12\min_{\ell\ne i}r_i(a_\ell)\right\}.
\]
Then $\{r_i<\varepsilon\}$ is a $\nu_r$-open neighborhood of $a_i$.
By continuity of $r_i$ and openness of $X\setminus F$, there is a
$\tau$-open neighborhood $W$ of $x$, contained in $X\setminus F$, on
which $r_i>\varepsilon$. Thus $W$ is $\nu_r$-open and disjoint from
$\{r_i<\varepsilon\}$. Finally, two distinct points of $X\setminus F$
can be separated inside the open Hausdorff subspace $X\setminus F$,
whose topology is unchanged. Hence $\nu_r$ is Hausdorff.
\end{proof}

For $s\in\GF$ we write $\tau_s=\nu_s$. It is immediate that
$\tau_s\subseteq\tau$ and that the two topologies agree on $X\setminus F$.
The next section shows that the exact Hausdorff condition above in fact
already guarantees metrizability, by an explicit metric.

\section{A simplicial metric}
Let
\[
 \Del=\left\{r=(r_1,\ldots,r_k)\in[0,1]^k:
                    \sum_{i=1}^k r_i=1\right\}
\]
be the standard simplex, with vertices $e_1,\ldots,e_k$. Put
\begin{equation}\label{eq:kappa}
 \kappa(r,t)=\frac12\sum_{i=1}^k|r_i-t_i|,
 \qquad H(r)=1-\max_{1\le i\le k}r_i.
\end{equation}
Thus
\[
 H(r)=\min_i\kappa(r,e_i),
\]
so $H$ is the $\kappa$-distance from the vertex set. In particular, $H$ is
$1$-Lipschitz and $H(r)=0$ exactly at the vertices.

The following elementary inequality is the metric core of the construction.
\begin{lemma}[Simplex slack inequality]\label{lem:slack}
Let $r,t,q\in\Del$ and put $m=\min\{H(r),H(q)\}$.
If $H(t)<m$, then
\begin{equation}\label{eq:slack}
 \kappa(r,t)+\kappa(t,q)-\kappa(r,q)\ge m-H(t).
\end{equation}
\end{lemma}
\begin{proof}
Choose $j$ such that $t_j=\max_i t_i=1-H(t)$.
Since $H(r)\ge m>H(t)$, every coordinate of $r$ is at most
$1-m<t_j$; in particular $r_j\le1-m$. The same holds for $q_j$.
The contribution of the $j$th coordinate to
$\kappa(r,t)+\kappa(t,q)-\kappa(r,q)$ is
\[
 \frac12\bigl(|r_j-t_j|+|t_j-q_j|-|r_j-q_j|\bigr)
 =t_j-\max\{r_j,q_j\}.
\]
Hence this contribution alone is at least
\[
 (1-H(t))-(1-m)=m-H(t).
\]
All remaining coordinate contributions are nonnegative by the triangle
inequality on $\R$.
\end{proof}

We now isolate a metric lemma that will also be useful independently of the
gauge construction.
\begin{lemma}[Simplicial warped metric]\label{lem:warped}
Let $(X,d)$ be a metric space with $d\le1$, and let
$U:X\to\Del$ be any map such that
\begin{equation}\label{eq:fibers}
 U^{-1}(e_i)\text{ contains at most one point for each }i.
\end{equation}
Define
\begin{equation}\label{eq:warped}
 \rho_U(x,y)=\kappa(U(x),U(y))
       +\min\{H(U(x)),H(U(y))\}\,d(x,y).
\end{equation}
Then $\rho_U$ is a metric on $X$.
\end{lemma}
\begin{proof}
Symmetry and nonnegativity are clear. If $\rho_U(x,y)=0$, then
$U(x)=U(y)$. If this common value is not a vertex, its $H$-value is positive,
so the second term in \eqref{eq:warped} gives $d(x,y)=0$.
If it is a vertex, \eqref{eq:fibers} again gives $x=y$.

It remains to prove the triangle inequality. Fix $x,y,z\in X$ and write
\[
 r=U(x),\qquad t=U(y),\qquad q=U(z),
\]
\[
 A=H(r),\qquad B=H(t),\qquad C=H(q),\qquad m=\min\{A,C\}.
\]
If $B\ge m$, then $\min\{A,B\}\ge m$ and $\min\{B,C\}\ge m$.
Using the triangle inequalities for $\kappa$ and $d$ gives
\[
 \rho_U(x,y)+\rho_U(y,z)\ge\kappa(r,q)+m\,d(x,z)=\rho_U(x,z).
\]
Suppose now that $B<m$. Then $\min\{A,B\}=\min\{B,C\}=B$.
By Lemma~\ref{lem:slack},
\[
 \kappa(r,t)+\kappa(t,q)\ge\kappa(r,q)+(m-B).
\]
Moreover, since $d\le1$,
\[
 \begin{split}
 m\,d(x,z)&=B\,d(x,z)+(m-B)d(x,z)\\
           &\le B\bigl(d(x,y)+d(y,z)\bigr)+(m-B).
 \end{split}
\]
Combining the last two inequalities yields
$\rho_U(x,z)\le\rho_U(x,y)+\rho_U(y,z)$.
\end{proof}

We now associate a simplex-valued map to a compatible gauge system.
For $s=(s_1,\ldots,s_k)\in\GF$, set
\begin{equation}\label{eq:products}
 \begin{gathered}
 P_i(x)=\prod_{j\ne i}s_j(x),\qquad
 u_i(x)=\frac{P_i(x)}{\sum_{\ell=1}^k P_\ell(x)},\\
 U_s(x)=(u_1(x),\ldots,u_k(x)).
 \end{gathered}
\end{equation}
The denominator is always positive: outside $F$ all $s_i$ are positive,
while at $a_i$ the single term $P_i(a_i)$ is positive and all the others
vanish. Hence $U_s$ is continuous and
\begin{equation}\label{eq:vertices}
 U_s^{-1}(e_i)=\{a_i\}.
\end{equation}
The key point is that the distance to $e_i$ recovers the local scale of
$s_i$ up to cofinal equivalence.

\begin{lemma}\label{lem:cofinal}
Fix $i$. The two families
\[
 \bigl\{\{s_i<\varepsilon\}:\varepsilon>0\bigr\}
 \quad\text{and}\quad
 \bigl\{\{1-u_i<\varepsilon\}:\varepsilon>0\bigr\}
\]
are mutually cofinal.
\end{lemma}
\begin{proof}
Let $c_i=\min_{j\ne i}c_{ij}(s)>0$.
If $x\notin F$, then from \eqref{eq:products}
\begin{equation}\label{eq:ratios}
 u_i(x)=\frac{1}{1+\displaystyle\sum_{j\ne i}\frac{s_i(x)}{s_j(x)}}.
\end{equation}
For $x\notin F$, if $s_i(x)<c_i$, compatibility gives $s_j(x)\ge c_i$
for every $j\ne i$. Therefore
\[
 1-u_i(x)\le\sum_{j\ne i}\frac{s_i(x)}{s_j(x)}
            \le\frac{k-1}{c_i}s_i(x).
\]
At $a_i$ the same implication is automatic, while at $a_j$ with $j\ne i$
the inequality $s_i(a_j)<c_i$ cannot occur, since
$s_i(a_j)\ge c_{ij}(s)\ge c_i$. Thus sufficiently small $s_i$ forces
arbitrarily small $1-u_i$ on all of $X$.

Conversely, since every $s_j\le1$, for $x\notin F$ the sum in
\eqref{eq:ratios} satisfies
\[
 R_i(x):=\sum_{j\ne i}\frac{s_i(x)}{s_j(x)}\ge(k-1)s_i(x).
\]
Also $1-u_i(x)=R_i(x)/(1+R_i(x))$. Hence, for $0<\eta<1$,
\[
 1-u_i(x)<\eta\quad\Longrightarrow\quad
 s_i(x)<\frac{\eta}{(1-\eta)(k-1)}.
\]
The exceptional points cause no difficulty: $1-u_i(a_i)=0$, whereas
$1-u_i(a_j)=1$ for $j\ne i$. The asserted cofinality follows.
\end{proof}

\begin{theorem}[Explicit simplicial metrization]\label{thm:metric}
Let $s\in\GF$ and let $d\le1$ be any compatible metric for $\tau$. Define
\begin{equation}\label{eq:metric}
 \rho_s(x,y)=\kappa(U_s(x),U_s(y))
       +\min\{H(U_s(x)),H(U_s(y))\}\,d(x,y).
\end{equation}
Then $\rho_s$ is a metric and $\tau_{\rho_s}=\tau_s\in\LF$.
More precisely,
\begin{equation}\label{eq:center-distance}
 \rho_s(a_i,x)=1-u_i(x),
\end{equation}
so the $\rho_s$-balls at $a_i$ are cofinal with the sublevel sets of $s_i$.
The topology is independent of the chosen compatible metric $d\le1$.
If $d$ is complete, then $\rho_s$ is complete.
\end{theorem}
\begin{proof}
The metric assertion follows from Lemma~\ref{lem:warped} and
\eqref{eq:vertices}. Since $H(e_i)=0$ and $\kappa(e_i,r)=1-r_i$
for $r\in\Del$, equation \eqref{eq:center-distance} follows.
Lemma~\ref{lem:cofinal} therefore shows that the local topology of $\rho_s$
at each $a_i$ is exactly the one prescribed by $s_i$.

It remains to compare the topologies away from $F$. The map $U_s$ is
$\tau$-continuous, so \eqref{eq:metric} shows that the identity
$(X,\tau)\to(X,\tau_{\rho_s})$ is continuous.
Now fix $x\notin F$ and put $h=H(U_s(x))>0$. If $\rho_s(x,y)<h/2$,
then $\kappa(U_s(x),U_s(y))<h/2$.
Since $H$ is $1$-Lipschitz, $H(U_s(y))>h/2$. Hence
\[
 \rho_s(x,y)\ge\frac h2 d(x,y),
\]
which proves continuity of the inverse identity at $x$. Thus the topologies
agree on $X\setminus F$, and the local description at the points of $F$
gives $\tau_{\rho_s}=\tau_s$. Independence of $d$ is immediate from this
description.

Assume finally that $d$ is complete and let $(x_n)$ be $\rho_s$-Cauchy.
Then $(U_s(x_n))$ is $\kappa$-Cauchy and hence converges to some
$r\in\Del$. If $H(r)=0$, then $r=e_i$ for some $i$, and
\eqref{eq:center-distance} gives
\[
 \rho_s(a_i,x_n)=1-u_i(x_n)\longrightarrow0.
\]
If $H(r)>0$, then eventually $H(U_s(x_n))\ge H(r)/2$, so
\[
 d(x_n,x_m)\le\frac{2}{H(r)}\rho_s(x_n,x_m)
\]
for all large $m,n$. Thus $(x_n)$ is $d$-Cauchy and converges to some
$x\in X$. Continuity of $U_s$ and the defining formula then give
$\rho_s(x_n,x)\to0$.
\end{proof}

\begin{corollary}[Hausdorffness and metrizability coincide for finite gauge topologies]
\label{cor:hausdorff-metric}
Let $r=(r_1,\ldots,r_k)$ be any finite gauge tuple as in
Proposition~\ref{prop:hausdorff}, and let $\nu_r$ be its gauge topology.
Then the following are equivalent:
\begin{enumerate}
\item $r$ is compatible;
\item $\nu_r$ is Hausdorff;
\item $\nu_r$ is metrizable.
\end{enumerate}
When these conditions hold, the simplicial formula \eqref{eq:metric},
with $s=r$, gives an explicit compatible metric.
\end{corollary}
\begin{proof}
Proposition~\ref{prop:hausdorff} gives (i)$\Leftrightarrow$(ii),
Theorem~\ref{thm:metric} gives (i)$\Rightarrow$(iii), and
(iii)$\Rightarrow$(ii) is automatic.
\end{proof}

\begin{remark}[The two-point bridge formula]
Let $F=\{a,b\}$ and write the compatible system as $(s_a,s_b)$. Put
\[
 u(x)=\frac{s_a(x)}{s_a(x)+s_b(x)}.
\]
Then $u(a)=0$, $u(b)=1$, and the sublevel sets of $u$ at $0$ are
cofinal with those of $s_a$, while the sublevel sets of $1-u$ are
cofinal with those of $s_b$. Identifying $\Delta^1$ with $[0,1]$,
formula \eqref{eq:metric} becomes
\[
 \begin{gathered}
 \rho_u(x,y)=|u(x)-u(y)|+\min\{h(u(x)),h(u(y))\}\,d(x,y),\\
 h(t)=\min\{t,1-t\}.
 \end{gathered}
\]
Thus the two-point theory may be encoded by a single bridge gauge.
\end{remark}

\section{Representation and order}
The preceding construction starts from a compatible system. We now show that
every member of $\LF$ arises in this way.

\begin{theorem}[Finite gauge representation]\label{thm:representation}
Assume $k\ge2$. Let $\sigma\in\LF$ and let $q$ be any compatible metric
for $\sigma$. Put
\begin{equation}\label{eq:normalized-distances}
 s_i(x)=\frac{q(x,a_i)}{1+q(x,a_i)}\qquad(i=1,\ldots,k).
\end{equation}
Then $s=(s_1,\ldots,s_k)$ belongs to $\GF$ and $\tau_s=\sigma$.
Consequently every finite-point localized metrizable coarsening admits an
explicit metric of the form \eqref{eq:metric}.
\end{theorem}
\begin{proof}
Each $s_i$ is $\sigma$-continuous and hence $\tau$-continuous because
$\sigma\subseteq\tau$. Its zero set is exactly $\{a_i\}$.
For $i\ne j$, set $c=q(a_i,a_j)>0$. For every $x\in X$, the triangle
inequality gives $q(x,a_i)+q(x,a_j)\ge c$, so at least one of these two
numbers is at least $c/2$. Since $t\mapsto t/(1+t)$ is increasing,
\[
 \max\{s_i(x),s_j(x)\}\ge\frac{c/2}{1+c/2}>0.
\]
Thus $s$ is compatible.

The sublevel sets of $s_i$ are precisely reparametrized $q$-balls at $a_i$,
so they form the $\sigma$-neighborhood base at $a_i$. The topologies
$\sigma$ and $\tau_s$ agree on $X\setminus F$ and have the same local
bases at every point of $F$. Hence they are equal. The particular
normalization $t\mapsto t/(1+t)$ is used only to place the gauges in
$[0,1]$; any increasing homeomorphism of $[0,\infty)$ onto $[0,1)$
would give an equivalent gauge system.
\end{proof}

\begin{remark}[Relation with extension traces]\label{rem:traces}
Suppose $Y=X\setminus F$ is dense in a localized coarsening $\sigma$.
A decreasing sequence of sufficiently small gauge sublevel sets at each
$a_i$, after deleting $F$, gives a local trace on $Y$.
The cofinal comparison below is the coordinatewise counterpart of
\cite[Theorems~3.5 and~3.9]{HJW}. The ambient coarsening constraint
in the one-point case is described in \cite[Proposition~4.1]{JS}.
Here the additional issue is simultaneous compatibility at the prescribed
centers, together with an explicit metric realizing all their bases.
\end{remark}

For $s,t\in\GF$ define
\begin{equation}\label{eq:preorder}
 s\preceq t
\end{equation}
if for every $i$ and every $\varepsilon>0$ there exists $\delta>0$ such that
\begin{equation}\label{eq:cofinal-order}
 \{t_i<\delta\}\subseteq\{s_i<\varepsilon\}.
\end{equation}
Write $s\asymp t$ when both $s\preceq t$ and $t\preceq s$.

\begin{theorem}[Order representation]\label{thm:order}
For $s,t\in\GF$,
\[
 \tau_s\subseteq\tau_t\quad\Longleftrightarrow\quad s\preceq t.
\]
Consequently,
\[
 \GF/{\asymp}\longrightarrow\LF,\qquad[s]\longmapsto\tau_s,
\]
is an order isomorphism.
\end{theorem}
\begin{proof}
Assume first that $\tau_s\subseteq\tau_t$. For each $i$ and
$\varepsilon>0$, the set $\{s_i<\varepsilon\}$ is a
$\tau_t$-neighborhood of $a_i$, so it contains a sublevel set
$\{t_i<\delta\}$ for some $\delta>0$. Thus $s\preceq t$.

Conversely, suppose $s\preceq t$ and let $U\in\tau_s$.
Away from $F$ the two topologies coincide. If $a_i\in U$, choose
$\varepsilon_i>0$ with $\{s_i<\varepsilon_i\}\subseteq U$, and then
choose $\delta_i>0$ with
\[
 \{t_i<\delta_i\}\subseteq\{s_i<\varepsilon_i\}.
\]
Thus every exceptional point belonging to $U$ is a $\tau_t$-interior
point, and $U\in\tau_t$. The remaining assertions follow from
Theorem~\ref{thm:representation}.
\end{proof}

\section{Joins, meets, and the obstruction to lower bounds}
The one-point family studied in \cite[Corollary~3.6]{JS} is a distributive
lattice. For a finite exceptional set, joins still behave coordinatewise,
but meets are controlled by compatibility.

\begin{theorem}[Finite joins]\label{thm:joins}
Let $s,t\in\GF$ and define $(s\vee t)_i=\max\{s_i,t_i\}$.
Then $s\vee t\in\GF$ and
\begin{equation}\label{eq:join}
 \tau_s\vee\tau_t=\tau_{s\vee t}.
\end{equation}
In particular, $\LF$ is closed under nonempty finite joins.
\end{theorem}
\begin{proof}
The zero set of $\max\{s_i,t_i\}$ is $\{a_i\}$. Moreover, for $i\ne j$,
\[
 \max\{\max(s_i,t_i),\max(s_j,t_j)\}\ge\max\{s_i,s_j\},
\]
so compatibility is preserved.
At $a_i$,
\[
 \{\max(s_i,t_i)<\varepsilon\}
   =\{s_i<\varepsilon\}\cap\{t_i<\varepsilon\}.
\]
These equal-radius intersections are cofinal among all intersections of an
$s_i$-sublevel set with a $t_i$-sublevel set. Hence they give exactly the
local base of the join topology at $a_i$. Away from $F$ there is nothing
to check.
\end{proof}

For a pair of systems define their coordinatewise minimum by
\[
 (s\wedge_0 t)_i=\min\{s_i,t_i\}.
\]
The notation $\wedge_0$ is used because this system need not be compatible.
Notice that each coordinate still has the correct singleton zero set;
only compatibility can fail. Hence the raw gauge topology $\nu_m$ is
always defined.

\begin{lemma}[Raw minimum identity]\label{lem:minimum}
For $s,t\in\GF$ and $m=s\wedge_0t$,
\begin{equation}\label{eq:raw-minimum}
 \tau_s\cap\tau_t=\nu_m.
\end{equation}
No compatibility assumption on $m$ is required.
\end{lemma}
\begin{proof}
Let $U\subseteq X$. If $U\in\tau_s\cap\tau_t$, then $U\in\tau$
and, for every $i$ with $a_i\in U$, there are $\varepsilon_i,\delta_i>0$
such that
\[
 \{s_i<\varepsilon_i\}\subseteq U,\qquad
 \{t_i<\delta_i\}\subseteq U.
\]
Putting $\eta_i=\min\{\varepsilon_i,\delta_i\}$ gives
\[
 \{m_i<\eta_i\}=\{s_i<\eta_i\}\cup\{t_i<\eta_i\}\subseteq U,
\]
so $U\in\nu_m$. Conversely, if $U\in\nu_m$ and $a_i\in U$,
then some $\{m_i<\eta_i\}$ lies in $U$; since this set contains both
$\{s_i<\eta_i\}$ and $\{t_i<\eta_i\}$, the set $U$ belongs to
both $\tau_s$ and $\tau_t$.
\end{proof}

\begin{lemma}[Crossed-sequence criterion]\label{lem:crossed}
Let $s,t\in\GF$ and put $m_i=\min\{s_i,t_i\}$. The minimum system is
incompatible if and only if there are distinct indices $i,j$ and a
sequence $(x_n)$ such that, after possibly interchanging $s$ and $t$,
\begin{equation}\label{eq:crossed}
 s_i(x_n)\longrightarrow0,\qquad t_j(x_n)\longrightarrow0.
\end{equation}
Equivalently, $x_n\to a_i$ in $\tau_s$ and $x_n\to a_j$ in $\tau_t$.
\end{lemma}
\begin{proof}
A crossed sequence as in \eqref{eq:crossed} gives $m_i(x_n)\to0$
and $m_j(x_n)\to0$, so the minimum system is incompatible.
Conversely, if it is incompatible, then for some $i\ne j$ there is
a sequence with $m_i(x_n)\to0$ and $m_j(x_n)\to0$.
For each $n$, each minimum is realized by one of its two entries.
Passing to a subsequence, one of the four realization patterns is
constant. The two non-crossed patterns are impossible by compatibility
of $s$ and $t$, respectively. Hence one of the crossed patterns remains.
The final equivalence follows from the sublevel neighborhood bases.
\end{proof}

\begin{theorem}[Meet criterion and failure of Hausdorff lower bounds]\label{thm:meet}
Let $s,t\in\GF$ and put $m=s\wedge_0t$, $m_i=\min\{s_i,t_i\}$.
The following are equivalent:
\begin{enumerate}
\item $m\in\GF$;
\item $\tau_s\cap\tau_t$ is Hausdorff;
\item $\tau_s$ and $\tau_t$ have a common Hausdorff lower bound;
\item $\tau_s$ and $\tau_t$ have a meet in $\LF$.
\end{enumerate}
When these conditions hold,
\begin{equation}\label{eq:meet}
 \tau_s\wedge\tau_t=\tau_m=\tau_s\cap\tau_t.
\end{equation}
If they fail, every topology on $X$ which is coarser than both $\tau_s$
and $\tau_t$ is non-Hausdorff.
\end{theorem}
\begin{proof}
By Lemma~\ref{lem:minimum}, $\tau_s\cap\tau_t=\nu_m$.
Proposition~\ref{prop:hausdorff} therefore gives
(i)$\Leftrightarrow$(ii). If (ii) holds, the intersection topology
itself is a common Hausdorff lower bound, so (ii)$\Rightarrow$(iii).
Conversely, if $\lambda$ is a common Hausdorff lower bound, then
$\lambda\subseteq\tau_s\cap\tau_t$. Every topology finer than a
Hausdorff topology is Hausdorff, so the intersection is Hausdorff;
hence (iii)$\Rightarrow$(ii).

If (i) holds, Theorem~\ref{thm:metric} shows that
$\nu_m=\tau_m$ belongs to $\LF$. Since it is the set-theoretic
intersection of the two topologies, it is their meet in $\LF$,
proving (i)$\Rightarrow$(iv) and \eqref{eq:meet}.
Finally, (iv)$\Rightarrow$(iii) because every member of $\LF$
is Hausdorff.

If the equivalent conditions fail and $\lambda$ is any common lower
topology, then $\lambda$ cannot be Hausdorff, for otherwise (iii)
would hold. Lemma~\ref{lem:crossed} supplies a concrete witness:
after relabeling, a sequence converges to distinct exceptional points
in the two original topologies, and hence to both points in every
common lower topology.
\end{proof}

\begin{corollary}\label{cor:existential}
For a one-point exceptional set, the family $\LF$ is a distributive lattice,
as proved in \cite[Corollary~3.6]{JS}. For every integer $k\ge2$, there exist
a metrizable space $(X,\tau)$, a set $F\subseteq X$ with $|F|=k$, and two
elements of $\LF$ having no common Hausdorff lower bound.
Thus the analogous meet property fails in general as soon as two
exceptional points are permitted.
\end{corollary}

\begin{example}[Crossed convergence on a discrete space]\label{ex:crossed}
Let $X=\{a,b\}\cup\{x_n:n\in\N\}$ with the discrete topology and let
$F=\{a,b\}$. Define two compatible systems $s=(s_a,s_b)$ and
$t=(t_a,t_b)$ by
\[
 \renewcommand{\arraystretch}{1.3}
 \begin{array}{c|ccc}
       &a&b&x_n\\ \hline
 s_a&0&1&\dfrac1{n+1}\\
 s_b&1&0&1\\
 t_a&0&1&1\\
 t_b&1&0&\dfrac1{n+1}
 \end{array}
\]
Then $x_n\to a$ in $\tau_s$ and $x_n\to b$ in $\tau_t$.
The minimum system satisfies
\[
 m_a(x_n)=m_b(x_n)=\frac1{n+1},
\]
so it is not compatible. By Theorem~\ref{thm:meet}, the two topologies
have no common Hausdorff lower bound.
For any $k>2$, take the topological sum of this example with $k-2$
isolated singleton spaces, add those new points to $F$, and extend both
coarsenings by the discrete topology on the added points.
Any common Hausdorff lower bound of the enlarged pair would restrict
to one for the original two-point example, which is impossible.
This also proves the $k\ge2$ assertion in Corollary~\ref{cor:existential}.
\end{example}

Theorem~\ref{thm:universal} below strengthens this existence statement:
the obstruction occurs for every noncompact metrizable space and every
finite exceptional set with at least two points.

\section{Strictness and compactness}
The one-point closed-discrete criterion
\cite[Theorem~4.3]{JS} extends directly to finite exceptional sets.

\begin{theorem}[Closed-discrete strictness criterion]\label{thm:strictness}
Let $s\in\GF$ and put $\sigma=\tau_s$. The following are equivalent
(with the convention $\inf\varnothing=+\infty$):
\begin{enumerate}
\item $\sigma\subsetneq\tau$;
\item for some $i$ there is a $\tau$-open neighborhood $U$ of $a_i$ such that
      \[
       \inf_{x\in X\setminus U}s_i(x)=0;
      \]
\item for some $i$ there is a countably infinite $\tau$-closed discrete set
      \[
       D=\{x_n:n\in\N\}\subseteq X\setminus F
      \]
      such that $x_n\to a_i$ in $\sigma$.
\end{enumerate}
\end{theorem}
\begin{proof}
The equality $\sigma=\tau$ holds exactly when, for every $i$, each
$\tau$-neighborhood $U$ of $a_i$ contains an $s_i$-sublevel set.
Since $s_i$ is positive on $X\setminus\{a_i\}$, this is equivalent to
$\inf_{X\setminus U}s_i>0$ for every such $U$.
This proves (i)$\Leftrightarrow$(ii).

Assume (ii). Shrinking $U$ if necessary, we may suppose that
$U\cap F=\{a_i\}$. This does not destroy the condition
$\inf_{X\setminus U}s_i=0$, because shrinking $U$ only enlarges its
complement. Choose $y_n\in X\setminus U$ with $s_i(y_n)<1/n$.
The range is infinite and each point occurs only finitely often,
because $s_i$ is strictly positive at every point of $X\setminus U$.
Passing to a subsequence, obtain distinct points $x_n\in X\setminus U$
with $s_i(x_n)\to0$. Eventually $x_n\notin F$, so we may delete finitely
many terms and assume $D=\{x_n:n\in\N\}\subseteq X\setminus F$.

Then $x_n\to a_i$ in $\sigma$. If $D$ had a $\tau$-accumulation
point $b$, metrizability of $\tau$ would give a subsequence converging
to $b$ in $\tau$, hence also in $\sigma$ because $\sigma\subseteq\tau$.
Since $b\in X\setminus U$, we have $b\ne a_i$, contradicting
uniqueness of limits in the Hausdorff space $(X,\sigma)$.
Thus $D$ is $\tau$-closed discrete, proving (iii).

If (iii) holds, the $\tau$-open neighborhood $X\setminus D$ of $a_i$
contains none of the $x_n$, so the sequence does not converge to $a_i$
in $\tau$. Hence $\sigma\ne\tau$.
\end{proof}

\begin{lemma}[A strict finite-point coarsening of a noncompact space]\label{lem:strict}
Let $(X,\tau)$ be noncompact and metrizable, and let
$F=\{a_1,\ldots,a_k\}$ be finite and nonempty.
Then $\LF$ contains a strict member.
\end{lemma}
\begin{proof}
Choose a compatible metric $d\le1$ for $\tau$. We first arrange an
auxiliary exceptional set with at least two points.
If $k\ge2$, put $E=F$. If $k=1$, choose $b\in X\setminus\{a_1\}$
and put $E=\{a_1,b\}$. Such a point exists because a noncompact
space is not a singleton. Write
\[
 E=\{b_1,\ldots,b_m\},\qquad b_1=a_1,\quad m\ge2.
\]
Since a metrizable space is compact if and only if it is sequentially
compact, there is a sequence with no convergent subsequence.
Its range is infinite and has no accumulation point: otherwise first
countability would produce a convergent subsequence.
After removing the finite set $E$ and passing to an infinite subset of
the range, we obtain a countably infinite $\tau$-closed discrete set
$D=\{x_n:n\in\N\}\subseteq X\setminus E$.
Put $\lambda_n=2^{-n}$ and define
\[
 h(x)=\inf_{n\ge1}\{\lambda_n+d(x,x_n)\},\qquad
 s_1(x)=\min\{d(x,b_1),h(x)\}.
\]
Each function $x\mapsto\lambda_n+d(x,x_n)$ is $1$-Lipschitz,
hence so is $h$. Moreover $h>0$ on $X$: if $x\notin D$,
closedness of $D$ gives $d(x,D)>0$, while for $x=x_m$
discreteness of $D$ gives $\eta_m>0$ with
$d(x_m,x_n)\ge\eta_m$ for $n\ne m$, and therefore
\[
 h(x_m)\ge\min\{\lambda_m,\eta_m\}>0.
\]
Thus $s_1^{-1}(0)=\{b_1\}$ and $s_1(x_n)\le\lambda_n\to0$.
For $j=2,\ldots,m$, set $s_j(x)=d(x,b_j)$.
Then every $s_j$ is continuous, takes values in $[0,1]$,
and vanishes exactly at $b_j$. The system $s=(s_1,\ldots,s_m)$
is compatible. Indeed, for $2\le i<j\le m$, the triangle inequality gives
\[
 \max\{s_i(x),s_j(x)\}\ge\frac12d(b_i,b_j)
\]
for every $x$. If compatibility failed for $s_1$ and some $s_j$,
$j\ge2$, there would be a sequence $(y_r)$ with $s_1(y_r)\to0$
and $d(y_r,b_j)\to0$; then $y_r\to b_j$ in $\tau$, and continuity
of $s_1$ would force $s_1(b_j)=0$, contrary to $b_j\ne b_1$.
Hence $s$ is compatible.

Apply Theorem~\ref{thm:metric} to the exceptional set $E$.
The resulting topology $\tau_s$ agrees with $\tau$ on $X\setminus E$.
At every $b_j$ with $j\ge2$, the gauge $s_j=d(\,\cdot\,,b_j)$
has the original metric balls as its sublevel sets, so the local
topology there is also exactly $\tau$. Consequently $\tau_s$ differs
from $\tau$ only possibly at $b_1=a_1$, and therefore $\tau_s\in\LF$
even in the case $k=1$.

Finally, $s_1(x_n)\to0$, so $x_n\to a_1$ in $\tau_s$.
This convergence does not hold in $\tau$, because the $\tau$-open
neighborhood $X\setminus D$ of $a_1$ contains no $x_n$.
Therefore $\tau_s\subsetneq\tau$.
\end{proof}

\begin{corollary}[Localized compactness characterization]\label{cor:compactness}
Let $(X,\tau)$ be metrizable and let $F\subseteq X$ be finite and nonempty.
The following are equivalent:
\begin{enumerate}
\item $(X,\tau)$ is noncompact;
\item $\LF$ contains a strict member.
\end{enumerate}
\end{corollary}
\begin{proof}
The implication (i)$\Rightarrow$(ii) is Lemma~\ref{lem:strict}.
Conversely, if $\tau$ is compact and $\sigma\in\LF$, then the identity
map $(X,\tau)\to(X,\sigma)$ is a continuous bijection from a compact
space onto a Hausdorff space, hence a homeomorphism.
Therefore $\sigma=\tau$.
\end{proof}

The same closed-discrete set can be used to force convergence to two
different prescribed centers in two different coarsenings.
This yields a universal form of the obstruction in Section~5.
It also follows from the prescribed-set version of
\cite[Theorem~1.1]{JS}; the proof below realizes it directly by compatible
finite gauge systems.

\begin{theorem}[Universal failure of Hausdorff lower bounds]\label{thm:universal}
Let $(X,\tau)$ be a noncompact metrizable space, let $F\subseteq X$
be finite with $|F|\ge2$, and let $a,b$ be any two distinct points of $F$.
There exist strict coarsenings $\sigma_a,\sigma_b\in\LF$ and a
countably infinite $\tau$-closed discrete set
$D=\{x_n:n\in\N\}\subseteq X\setminus F$ such that
\[
 \sigma_a|_{X\setminus\{a\}}=\tau|_{X\setminus\{a\}},\qquad
 \sigma_b|_{X\setminus\{b\}}=\tau|_{X\setminus\{b\}},
\]
and
\[
 x_n\longrightarrow a\text{ in }\sigma_a,\qquad
 x_n\longrightarrow b\text{ in }\sigma_b.
\]
In particular, every topology on $X$ coarser than both $\sigma_a$
and $\sigma_b$ is non-Hausdorff.
\end{theorem}
\begin{proof}
Relabel $F=\{a_1,\ldots,a_k\}$ so that $a_1=a$ and $a_2=b$.
Choose a compatible metric $d\le1$. As in the proof of
Lemma~\ref{lem:strict}, noncompactness supplies a countably infinite
$\tau$-closed discrete set $D=\{x_n:n\in\N\}\subseteq X\setminus F$.
The function
\[
 h(x)=\inf_{n\ge1}\{2^{-n}+d(x,x_n)\}
\]
is $1$-Lipschitz and strictly positive on $X$, and $h(x_n)\le2^{-n}$.
Put $d_i(x)=d(x,a_i)$ and define
\[
 s_i(x)=
 \begin{cases}
  \min\{d_1(x),h(x)\},&i=1,\\
  d_i(x),&i\ne1,
 \end{cases}
 \qquad
 t_i(x)=
 \begin{cases}
  \min\{d_2(x),h(x)\},&i=2,\\
  d_i(x),&i\ne2.
 \end{cases}
\]
All these functions are continuous, take values in $[0,1]$, and have
the required singleton zero sets.

We check compatibility of $s$. If $i,j\ne1$ are distinct, then
$\max\{s_i(x),s_j(x)\}\ge d(a_i,a_j)/2$.
If compatibility failed for $s_1$ and $s_j$ with $j\ne1$,
there would be a sequence $(y_r)$ with $s_1(y_r)\to0$ and
$d(y_r,a_j)\to0$. Continuity of $s_1$ would imply
$s_1(a_j)=0$, whereas
$s_1(a_j)=\min\{d(a_j,a_1),h(a_j)\}>0$.
Thus $s$ is compatible. The same argument, with index $2$ in
place of $1$, proves compatibility of $t$.

By Theorem~\ref{thm:metric}, $\sigma_a:=\tau_s$ and
$\sigma_b:=\tau_t$ belong to $\LF$.
At every $a_j\ne a_1$, the $s_j$-sublevel sets are the original
$d$-balls; hence $\sigma_a$ agrees with $\tau$ off $a_1$.
Likewise $\sigma_b$ agrees with $\tau$ off $a_2$.
Moreover,
\[
 s_1(x_n)\le2^{-n},\qquad t_2(x_n)\le2^{-n},
\]
which gives the asserted convergences. Both coarsenings are strict,
because $X\setminus D$ is a $\tau$-open neighborhood of each center
and contains no term of the sequence.

If $\lambda\subseteq\sigma_a\cap\sigma_b$ is any topology on $X$,
the sequence $(x_n)$ converges to both $a_1$ and $a_2$ in $\lambda$.
Since these points are distinct, $\lambda$ is not Hausdorff.
Equivalently, this is the crossed-sequence obstruction of
Lemma~\ref{lem:crossed} and Theorem~\ref{thm:meet}.
\end{proof}

\begin{corollary}[Compactness--lattice dichotomy]\label{cor:dichotomy}
Let $(X,\tau)$ be metrizable and let $F\subseteq X$ be finite with
$|F|\ge2$. The following are equivalent:
\begin{enumerate}
\item $(X,\tau)$ is compact;
\item $\LF=\{\tau\}$;
\item every two members of $\LF$ have a common Hausdorff lower
      topology on $X$;
\item $\LF$ is downward directed, that is, every two members have
      a common lower bound belonging to $\LF$;
\item $\LF$ is a lattice.
\end{enumerate}
When these conditions hold, $\LF$ is the one-element distributive lattice.
\end{corollary}
\begin{proof}
The compact-to-Hausdorff argument in Corollary~\ref{cor:compactness}
gives (i)$\Rightarrow$(ii). A singleton is a lattice, so
(ii)$\Rightarrow$(v). In a lattice a meet is a common lower bound,
giving (v)$\Rightarrow$(iv). Every member of $\LF$ is Hausdorff,
so (iv)$\Rightarrow$(iii). Finally, if $(X,\tau)$ were noncompact,
Theorem~\ref{thm:universal} would contradict (iii).
Thus (iii)$\Rightarrow$(i), completing the cycle.
\end{proof}

The hypothesis $|F|\ge2$ is essential: for a one-point exceptional set
the family is a distributive lattice even when $(X,\tau)$ is noncompact,
as recalled in Corollary~\ref{cor:existential}.

\section{Preservation under a finite change}
The preservation consequences in \cite[Proposition~4.5]{JS} extend
to any finite exceptional set, without assuming inclusion between the
topologies. The completeness proof below is the compact-remainder
argument of \cite[Theorems~2.1--2.2]{HJW}, also covering a punctured
subspace that is not dense. These are standard supporting consequences.

\begin{proposition}\label{prop:preservation}
Let $\tau$ and $\sigma$ be metrizable topologies on $X$ and let
$F\subseteq X$ be finite. Suppose
$\tau|_{X\setminus F}=\sigma|_{X\setminus F}$. Then:
\begin{enumerate}
\item the two topologies have the same Borel subsets of $X$;
\item $(X,\tau)$ is completely metrizable if and only if $(X,\sigma)$
      is completely metrizable;
\item $(X,\tau)$ is Polish if and only if $(X,\sigma)$ is Polish.
\end{enumerate}
\end{proposition}
\begin{proof}
Put $Y=X\setminus F$. Since $F$ is finite and the spaces are
metrizable, $Y$ is open in both topologies and has the same
subspace topology. If $U\in\tau$, then $U\cap Y$ is $\sigma$-open,
while $U\cap F$ is finite and therefore $\sigma$-Borel.
Hence $U$ is $\sigma$-Borel. By symmetry the Borel structures coincide.

Suppose $(X,\tau)$ is completely metrizable. Then the open subspace
$Y$ is completely metrizable. Embed $(X,\sigma)$ into the completion
$Z$ of any compatible metric. By the standard $G_\delta$
characterization of complete metrizability (see, e.g., \cite{Engelking}),
$Y$ is $G_\delta$ in $Z$. Every singleton of the metric space $Z$
is $G_\delta$, hence the finite set $F$ is $G_\delta$.
Finite unions of $G_\delta$ sets are $G_\delta$, so $X=Y\cup F$
is $G_\delta$ in $Z$ and therefore completely metrizable.
Symmetry gives the converse.

If $(X,\tau)$ is Polish, then $Y$ is separable. A countable dense
subset of $Y$, together with the finite set $F$, is dense in
$(X,\sigma)$. Combining separability with (ii) proves that
$(X,\sigma)$ is Polish. Again the converse is symmetric.
\end{proof}

\section{Further directions}
The results above describe the separation obstruction and characterize
exactly when the full finite-point family is a lattice.
Several problems remain natural.
\begin{enumerate}[label=\textup{(\arabic*)}]
\item \textbf{Local metric preservation.}
Characterize those $\sigma\in\LF$ that admit a compatible metric
$p\le d$ which agrees exactly with a prescribed compatible metric $d$
on a common neighborhood of every point of $X\setminus F$.

\item \textbf{Extremal metrics with several centers.}
Develop a finite-point analogue of the greatest-metric construction
from \cite[Proposition~2.3]{JS}, allowing prescribed upper bounds on
the distances from several exceptional points to selected
closed-discrete sets while keeping the exceptional points mutually separated.

\item \textbf{Infinite exceptional sets.}
Determine which parts of the compatible-gauge representation survive
when $F$ is countably infinite or closed discrete. In particular,
determine whether pairwise compatibility suffices, or whether an
additional collective condition is needed to control the local bases
simultaneously.

\item \textbf{Partial lattice structure.}
Theorem~\ref{thm:meet} identifies exactly when a pair has a meet,
and Corollary~\ref{cor:dichotomy} shows that the full family is a
lattice only in the compact case when $|F|\ge2$.
It is natural to study maximal subfamilies of $\LF$ on which all
finite meets exist, and to determine whether useful distributive or
domain-theoretic structures arise there.
\end{enumerate}

\section*{AI Use Declaration}
The authors declare that GPT-5.6 Sol was used solely for language polishing and editorial assistance. All mathematical ideas, proofs, and results presented in this manuscript are entirely the work of the authors.

\end{document}